\documentclass[11pt]{article}
\usepackage{a4wide}
\usepackage{setspace}
\usepackage[]{graphicx,float,latexsym,times}
\usepackage{amstext,amsmath,amssymb,amsthm}
\usepackage{caption2}
\usepackage[authoryear]{natbib}

\long\def\symbolfootnote[#1]#2{\begingroup%
\def\thefootnote{\fnsymbol{footnote}}\footnote[#1]{#2}\endgroup}

\theoremstyle{plain} 
\newtheorem{theorem}{Theorem}[section]
\newtheorem{lemma}{Lemma}[section]
\newtheorem{corollary}{Corollary}[section]

\theoremstyle{definition} 

\newtheorem{remark}{Remark}[section]

\numberwithin{equation}{section} 

\newcommand{\N}{ \mathbb{N} }

\newcommand{\Z}{ \mathbb{Z} }

\newcommand{\R}{ \mathbb{R} }

\newcommand{\ito}{It\^o }

\newcommand{\trunc}[1]{ {\lfloor #1 \rfloor} }
\newcommand{\wh}[1]{ \widehat{ #1 } }
\newcommand{\wt}[1]{ \widetilde{ #1 } }

\newcommand{\calF}{\mathcal{F}}

\newcommand{\calH}{\mathcal{H}}

\newcommand{\calL}{\mathcal{L}}

\newcommand{\eins}{{ 1}}

\newcommand{\Var}{{\mbox{Var\,}}}

\newcommand{\diag}{{\mbox{diag\,}}}

\newcommand{\argmin}{ \operatorname{argmin} }

\newcommand{\id}{\operatorname{id}}

\newcommand{\cadlag}{c\`adl\`ag }

\begin{document}

\title{\Large Sequential Cross-Validated Bandwidth Selection Under Dependence and Anscombe-Type Extensions
to Random Time Horizons}

\date{}

\maketitle


\author{
\begin{center}
\vskip -1cm

\textbf{\large Ansgar \ Steland}

Institute of Statistics, RWTH Aachen University, \\
W\"ullnerstr.3, 52065 Aachen, Germany

\end{center}
}

{\small \noindent\textbf{Abstract:} To detect changes in the mean of a time series, one may use previsible detection procedures based on
nonparametric kernel prediction smoothers which cover various classic detection statistics as special cases. 
Bandwidth selection, particularly in a data-adaptive way, is a serious issue and not well studied for detection problems. 
To ensure data adaptation, we select the bandwidth by cross-validation, but in a sequential way leading to a 
functional estimation approach. This article provides the asymptotic theory for the method under fairly weak assumptions 
on the dependence structure of the error terms, which cover, e.g., GARCH($p,q$) processes, by establishing (sequential) functional central limit theorems for the cross-validation objective function and the associated bandwidth selector. It turns out that the proof can be based in a neat way on \cite{KurtzProtter1996}'s results on the weak convergence of \ito integrals and a diagonal argument.

Our gradual change-point model covers multiple change-points in that it allows for a nonlinear regression function after the first 
change-point possibly with further jumps and Lipschitz continuous between those discontinuities. 

In applications, the time horizon where monitoring stops latest is often determined by a random experiment, e.g.
a first-exit stopping time applied to a cumulated cost process or a risk measure, possibly stochastically dependent from the monitored
time series. Thus, we also study that case and establish related limit theorems in the spirit of \citet{Anscombe1952}'s result.
This is achieved by embedding  the stopped processes into a sequence of processes, which allows us to handle the randomly determined
time horizon as a random change of time problem. The result has various applications including statistical parameter estimation and monitoring financial investment strategies with risk-controlled early termination, which are briefly discussed.
}

\section{INTRODUCTION} \label{s:Intro}

Assuming a nonparametric regression model with non-vanishing mean, we study the sequential asymptotic distribution theory of the cross-validated
bandwidth selector for a previsible  kernel detection statistic.
The nonparametric regression model assumes that the mean of a process observed in discrete time is given by an unknown function belonging to
some infinite dimensional function space, providing an attractive framework for statistical estimation as well as detection problems. Estimation and inference based
on an observed (large) sample has been extensively studied in the literature. For an overview on this topic and the most common methods 
 such as kernel estimators, local polynomials, smoothing splines and wavelets, we refer to \citet{DonohoJohnston1994}, \cite{Eubank1988}, 
\cite{HaerdleAppliedNonpR1991} and \citet{WandJones1995} and the references given therein.

However, often the data arrive sequentially and interest is in detecting changes in the mean function, for instance that the mean is too large or too small. There is growing interest in sequential methods due to their importance for the analysis of data streams in areas such as finance, environemetrics and engineering.
Procedures based on nonparametric smoothers form an attractive class of methods, which has received substantial interest in the literature; we refer to
\citet{WuChu93}, \citet{MuellerStadtmueller99}, \citet{Steland2005JSPI} and \citet{Steland2010JNonpStat}, amongst others. Now invariance principles provide a neat
way to obtain distributional approximations under weak assumptions. For a discussion of that approach we refer to \citet{Steland2010c}.
There is also a rich literature on the estimation of regression functions that are smooth except some discontinuity (change-) points. See, for example, the recent work of \citet{GijbelsGoderniaux2004} or \citet{AntochGregoireHuskova2007}.

The problem how to select the bandwidth or, more generally, tuning constants, is a serious issue, which is not well understood for the sequential detection problem. Here,
asymptotic results for methods such as the CUSUM, MOSUM or EWMA procedures, studied by \citet{HanEtAl2004}, \citet{BrodskyDarkhovsky2008}, \citet{AueHorvathEtAl2008}, and \citet{Moustakides2008}  amongst many others, assume that the bandwidth parameter may depend on some size parameter, usually the time horizon or maximal sample size, but it is chosen in a non-stochastic (deterministic) way; a notable exception is the work of \citet{Spokoiny1998} and \citet{SpokoinyPolzehl2006}.
For a novel consistent approach to the selection of such tuning constants using singular value decomposition methods see
\citet{GPS2012}.
In  \citet{Steland2010a} we  proposed to use cross-validation, a general technique widely employed,  in a sequential way to a sequence of kernel prediction statistics related to the well known Nadaraya-Watson estimator. Using a sequence of prediction statistics forming a previsible process and being close to the Nadaraya-Watson estimator has the advantage that detection is based on a control statistic which can be interpreted as a one-step ahead predictor providing sequential approximations to the process mean. Contrary, various classic methods for change detection lack that nice interpretation.

In the sequential approach, the cross-validated bandwidth is, in principle, calculated at each time point thus yielding a functional bandwidth estimate. \citet{Steland2010a}
establishes uniform laws of  large numbers for the corresponding objective function as well as consistency theorems for the bandwidth estimate.
To guarantee wide applicability, those results are proved for i.i.d. data as well as for $L_2$-NED processes. 
The present article focuses on the relevant asymptotic distribution theory in the sense of weak convergence. We provide sequential functional central limit theorems for the cross-validation criterion as
well as for the functional estimate of the cross-validated bandwidth. An important tool for our theoretical results is a general result on the weak convergence of It\^{o} integrals 
for integrators which are semimartingales. We show that the results hold true under a weak $\alpha$-mixing condition which is satisfied by many processes, e.g. many linear processes, which are also known to be $S$-mixing, a class of processes for which \citet{BerkesHormannEtAl2009} recently established a strong invariance principles for the classic sequential empirical process. For related results for long-memory processes we refer to  \citet{DehlingTaqqu1989} and the work of
\citet{DoukhanEtAl2005}  which allows for a weakly dependent nonlinear Bernoulli shift component.
Further results can be found in \citet{DehlingMikosch2002}.

In a first step, we assume that the observations arrive sequentially until a non-random time horizon $ T \to \infty $ is reached.
By rescaling time to the unit interval, the Skorohod spaces of right-continuous function with left-hand limits provide an appropriate
framework to establish a weak limit theory. However, in certain applications the time horizon is not fixed but determined 
by a parameterized random experiment such as a family of random first exit stopping times. The
question arises under which conditions on those stopping times the stopped process inherits the asymptotic distribution. 
Results of this type can be traced back to the
seminal work of \citet{Anscombe1952}, which studied the large-sample theory of randomly stopped 
stochastic processes in discrete time. Thus, in a second step, we show that in our framework an embedding argument allows us to interpret
the randomly selected time horizon as a random change of time problem leading to a Anscombe-type theorem. We assume the same
condition on the family of random indices replacing the time horizon $T$ as imposed by Anscombe.

The sequential setup is as follows:
We assume that observations $ Y_n = Y_{Tn} $, $ 1 \le n \le T $, arrive sequentially until the maximum sample size $ T $ is reached and satisfy the model equation
\begin{equation}
\label{ModelPaper}
  Y_n = m(x_n) + \epsilon_n, \qquad n = 1, 2, \ldots, T, \ T \ge 1,
\end{equation}
with
\begin{equation}
\label{ModelForM}
  m(x_n) = m_0(x_n) + \delta(x_n)/\sqrt{T}.
\end{equation}
The time horizon $T$ is assumed to be non-random and large; it will converge to $ \infty $ in our limit theorems. Extensions to
random time horizons are discussed in Section~\ref{Sec:Anscombe}.
The function $ m_0 $ is assumed to be known. $ \delta $ is a bounded and piecewise Lipschitz continuous function on $[0,\infty)$ with at most finitely many jumps, either $ \delta>0 $ or $ \delta<0$, and such that 
 \[ q_1 = \inf \{ s > 0 : \delta(s) \not= 0 \} > \gamma \] 
 for some $ \gamma \in (0,1) $. In detection, the primary goal is to detect changes from an assumed model, $m_0$, for the process, also called {\em in-control model} or {\em null model}. The departure from that in-control or normal behavior is modeled by the function $ \delta $. Of particular interest is
 the detection of the first change point $ q_1 $. When $ \delta $ is a smooth function, the above model is also called gradual change model, since then the process mean smoothly drifts away from the assumed in-control behavior. But because we allow for $ \delta $-functions with jumps, (\ref{ModelForM}) is very general and covers the case that there are many change-points where the mean changes abruptly, e.g. when  $ \delta $ is a step function representing a finite number of level shifts. Hence, we treat a large class of change-points models in an unified way. We consider a sequence of local alternatives converging to the null model at the rate $ T^{-1/2} $, which will allow us to establish weak limits for the quantities of interest providing a means to study local performance properties, e.g. by simulating from the limit process. Of substantial interest is the detection of the first change-point of $ \delta $ after some initial time instance $s_0 $ where the monitoring procedure starts, i.e. $ \inf \{ s > s_0 : \delta(s) > 0 \} $, respectively.

For the regressors $ \{ x_n \} $ a fixed design
\[
  x_n = x_{Tn} = G^{-1}( n/T ), \qquad 1 \le n \le T,
\]
induced by some design distribution function $ G $ is assumed. In many applications one can assume that $G$ is known or chosen by the statistician. Examples cover biostatistical dose-response studies, applications in communication engineering with equidistant sampling as well as laboratory experiments where the design points are selected according to some external
criterion, cf. also \citet{Steland2010a}. For simplicity of our exposition, we will assume that $ G = \id $, since  otherwise one may substitute $m_0 $ by $ m_0 \circ G^{-1} $ and $ \delta $ by $ \delta \circ G^{-1} $. The term $ \delta_T = \delta/\sqrt{T} $ in (\ref{ModelPaper}) represents the local alternative model describing the departure from $m_0$.

Our results work under the weak assumption that the errors $ \{ \epsilon_t \} $ form a strictly stationary martingale difference sequence satisfying a classic condition on the strong mixing coefficients. Some of our results even hold true for stationary martingale difference sequences without additional assumptions.

The organization of the paper is as follows. In Section~\ref{Sec:SeqCrossVal}, we introduce the sequential cross-validation approach. Section~\ref{sec:Preliminaries} provides some basic notation and preliminaries as well as an exposition of a result on the weak convergence of \ito integrals, which we shall use to prove the results. The main asymptotic results are given and proved in Section~\ref{Sec:Asymptotics}.
Section~\ref{Sec:Anscombe} discusses the extension to random time horizons, its relationship to Anscombe's classical result and our
Anscombe-type result based on a random change of time argument.

\section{SEQUENTIAL CROSS-VALIDATION}
\label{Sec:SeqCrossVal}

The statistical idea of cross-validation is to choose nuisance parameters such as tuning constants controlling the degree of smoothing of a statistic in a data-adaptive way such that the corresponding estimates provide a good fit on average. Let us define the sequential, i.e.,
$ \calF_{i-1} = \sigma( Y_j: 1 \le j \le i-1 ) $-measureable prediction estimate
\begin{equation}
\label{DefWHat}
  \widehat{m}_{h,-i} = N_{T,-i}^{-1} \frac{1}{ h } \sum_{j=\trunc{T\gamma}}^{i-1} K( [j-i]/h ) Y_j, \qquad i = \trunc{T\gamma}, \trunc{T\gamma}+1, \ldots
\end{equation}
where $ N_{T,-i} = h^{-1} \sum_{j=\trunc{T\gamma}}^{i-1} K((j-i)/h) $ and $ \gamma \in (0,1) $ is an arbitrary small but fixed constant. 
$K$ is a kernel function
such that
\begin{equation}
\label{ST:AssumptionsK}
 \text{$K \in $ Lip($[0,\infty); [0,\infty)$), $ \| K \|_\infty < \infty $  and $K>0$,} 
\end{equation}
where $ \text{Lip($[0,\infty);\R$)} $ denotes the class of Lipschitz continuous functions on $ [0,\infty) $. We assume that the bandwidth $h > 0$ is a function of the time horizon $T$
in such a way that 
\begin{equation}
\label{ST:Bandwidth}
  |T/h - \xi | = O( 1/T ) 
\end{equation}
for some constant $ \xi \in ( 0, \infty ) $. Imposing the convergence rate $ T^{-1} $ rules out artificial choices such as $ h = T/(\xi + T^{-\gamma}) $, $ \gamma > 0 $, leading to arbitrary slow convergence. 

To this end, let $ \calF_n $ be the natural filtration associated to $ \{ \epsilon_n \} $. Substituting $ h $  in  $ \widehat{m}_{h,-i} $ by a row-wise $ \calF_i $-adapted array 
$ h_{Ti}^* $, $ \trunc{T s_0} \le i \le T $, $ T \ge 1 $, of non-negative random variables yields again an adapted array $ \{ \widehat{m}_{h_{Ti}^*,-i} \} $ to which
we apply one-sided detection procedures given by the first exit stopping times 
\begin{align}
\label{STplus}
  S_T^+ &= \inf \{ \trunc{s_0 T} \le i \le T : \wh{m}_{h^*_{Ti},-i} > c \},\\
\label{STminus}
  S_T^-  &= \inf \{ \trunc{s_0 T} \le i \le T : \wh{m}_{h^*_{Ti},-i} < c \},
\end{align}
respectively, where $ s_0 > \gamma $ determines the start of monitoring. 
Given the predictions $ \widehat{m}_{h,-i}  $, we may define the {\em sequential leave-one-out cross-validation criterion} 
\[
  CV_s(h) = CV_{T,s}(h) = \frac{1}{T} \sum_{i=\trunc{Ts_0}}^{\trunc{Ts}} ( Y_i - \widehat{m}_{h,-i} )^2, \qquad h > 0,
\]
a function of the candidate bandwidth $h$.
In the functional cross-validation bandwidth approach the cross-validation objective function is minimized for each $ s \in [s_0,1] $. To do so,
let $ \calH_{s_0,\xi} $ be the family of all arrays $ \{ h_{Tn} : \trunc{s_0T} \le n \le T, \ T \ge 1 \} $ with
\[
 \max_{1 \le n \le T} |T / h_{Tn} - \xi | = O(1/T) \qquad \text{for some $ \xi > 0 $}.
\]
We consider minimizers $ \{ h_{Tn}^* \} \in \calH_{s_0,\xi} $ of the cross-validation criterion such that 
\[
  CV_{n/T}(h_{Tn}^*) \le CV_{n/T}( h_{Tn} ), \qquad \trunc{s_0T} \le n \le T, \ T \ge 1, 
\]
for all $ \{ h_{Tn} \} \in \calH_{s_0,\xi} $. This leads to the functional cross-validated bandwidth estimator
\[
  h^*_T(s) = h_{T,\trunc{Ts}}^*, \qquad s \in [s_0,1].
\]
Notice that, by definition, $ \{ h^*_T(s) : s \in [s_0,1] \} $ is $ \calF_{\trunc{Ts}} $-adapted. 
\citet{Steland2010a} showed that, under regularity assumptions, $ CV_{T,s}(h) $ converges to some function $ CV_\xi(s) $ which depends on $ \xi = \lim T/h $. That result is valid for stationary $ \alpha $-mixing series and for $L_2$-near epoch dependent time series. For i.i.d. error terms one even achieves the usual $ O(1/\sqrt{T}) $ rate of convergence in the sense of $ L_2 $ convergence. Having those results in mind, we now address the related weak convergence theory.

Since in practice the cross-validation criterion has to be minimized numerically, one may assume that minimization is done over a finite grid of values, and
we shall provide a weak convergence result for the cross-validated bandwidth under such an assumption.
Further, conducting cross-validation at each time point can be infeasible in a practical application, such that one has to select $N$ time points, $ s_0 \le s_1 < \cdots < s_N $, 
where the cross validation criterion is numerically minimized, thus yielding an adapted sequence
$ h_i^* = h_T^*(s_i) $, $i = 1, \dots, N $. The cross-validated bandwidth $ h_i^* $ is then used during the time interval $ [s_i,s_{i+1}) $, $ i = 1, \dots, N $.
Clearly, the corresponding cross-validation bandwidth estimator is now the step function
\[
  h^*_{TN}(s) = h_T^*(s_i), \qquad s \in [s_i,s_{i+1}), \ i = 1, \dots, N-1.
\]
In such a situation, it is sufficient to know the convergence of the finite dimensional distributions (fidi convergence).

\section{PRELIMINARIES AND WEAK CONVERGENCE OF STOCHASTIC INTEGRALS}
\label{sec:Preliminaries}

Since $ L_T $ and $ Q_T $ are random \cadlag functions, it is in order to recall some basic facts on the Skorohod spaces $ D([a,b];\R^l) $, $l$ an
integer, consisting of those functions $ [a,b] \to \R^l $, $ a, b \in \R $, being right-continuous with existing limits from the left. Let
$ V(f) $ denote the (total) variation semi-norm of a function $f$ and $ \| f \|_\infty $ its supnorm. 
For a random variable $ X $ we denote by $ \| X \|_p $ the $L_p $-norm, $ p \in [0,\infty) $. The space $ D([a,b] ; \R^l ) $ can be equipped with the following
Skorohod metric. For two functions $ f, g $ on $ [a,b] $ with values in $ \R^l $ define
\[
  d( f, g ) = \inf_{ \lambda \in \Lambda } \max \{ \| f \circ \lambda - g \|_\infty, \| \lambda - \id \|_\infty \},
\]
where $ \Lambda $ is the set of all strictly increasing continuous mappings $ \lambda : [a,b] \to [a,b] $.  
Clearly, $ d(f,g) \le \| f - g \|_\infty $, such that uniform convergence
implies convergence in the Skorohod metric. Weak convergence of a sequence $ \{ X, X_n \} $ of random functions taking values in $ D([a,b];\R^l ) $ 
now means weak convergence of the measures $ P_{X_n} $ to $ P_X $, as $ n \to \infty $, denoted by $ X_n \Rightarrow X $,
$n \to \infty $. For the sake of clarity of exposition, we shall also write $ X_n(u) \Rightarrow X(u) $, as $ n \to \infty $.
Further details can be found in \citet{BickelWichura71}, \citet{Neuhaus71}, \citet{Straf72} and \citet{SeijoSen2011}.

The framework for the weak convergence result for It\^o integrals is as follows. Let us first recall the definition of the It\^{o} integral, cf.
\citet{Protter05} or \citet{Steland2012Book}. Let $ \{ H_n \} $ and $ \{ X_n \} $ be sequences
of adapted processes on a probability space $ (\Omega, \calF, P)$ which is equiped with a sequence $ \{ \calF_n \} $ of filtrations
$ \calF_n = \{ \calF_{nt} : t \in I  \}  $  with index set $I$, i.e. $ H_n, X_n $ are $ \calF_{nt} $-adapted such that
$ H_n(t), X_n(t) $ are $ \calF_{nt} $-measureable, $ t \in I $. In general, a process $ X $ is called a semimartingale, if $ X = M + A $ for some local martingale and a
process $A $ having bounded variation. Given a semimartingale $X$ and a predictable \cadlag process $H$, one may define the stochastic It\^o integral 
\[ 
  \int H \, d X = \left\{ \int_0^t H(s-) \, d X(s) : t \in I \right\}.
\]
When we equip the space $L(I;\R) $ of left continuous functions possessing right-hand limits with the topology induced by the uniform convergence on
compact sets, the linear operator $ I(\cdot) = \int \cdot \, d X $ is continuous on $L(I;\R)$, such that uniform convergence $ H_n \to H $ on compact sets
of a sequence $ \{ H, H_n \} $ of such adapted processes implies convergence of the It\^{o} integral, in probability, and therefore also weakly.
The following result extends the latter fact to the much more involved case that the integrator depends on $n$.

\begin{theorem} (Kurtz and Protter, 1996)\\
\label{KurtzProtter}
Suppose that $ X_n $ is, for each $n \in \N$, a $ \calF_{nt} $-adapted semimartingale with Doob decomposition $ X_n = M_n + A_n $ such that $ \sup_n \Var(X_n) + V(A_n)< \infty $, and $ H_n $ is $ \calF_{nt} $-predictable. If $ (H_n,X_n) \Rightarrow (H,X) $, as $n \to \infty $, in the Skorohod space $ D([a,b];\R^2) $, then
\[
  \left( H_n, X_n, \int H_n \, d X_n \right) \Rightarrow \left( H, X, \int H \, d X \right),
\]
as $ n \to \infty $, in $ D([a,b]; \R^3) $.
\end{theorem}
We will apply that result to the following framework. Assume that the $ \epsilon_n $ are defined on a common probability space
$ (\Omega, \calF, P) $ which we equip with a sequence of filtrations $ \calF_{nt}  $. For simplicity, one may consider the natural filtrations
$ \calF_{nt} = \sigma( \epsilon_i : i \le \trunc{nt} ) $, $t \in [s_0,1] $, $ n \in \N $, in what follows, but the results hold true for any
 sequence of filtrations such that $ \epsilon_n $ is $ \calF_n = \calF_{n,1} $ adapted. Our assumptions on $ \delta $
ensure that $ t \mapsto  \int_{(0,t]} \delta \, d \lambda $, $ \lambda $ denoting Lebesgue measure, exists and defines a function of bounded variation.

\begin{lemma} 
\label{FCTLPartialSum}
Suppose that $ \{ \epsilon_t \} $ is a $ \calF_n $-martingale difference sequence under $P$. Then the partial sum process
\begin{equation}
\label{DefPartialSum}
S_T(u) = T^{-1/2} \sum_{i=1}^{\trunc{Tu}} Y_i, \qquad u \in [0,1], T \ge 1,
\end{equation}
defines a sequence of semimartingales. If, additionally, $ \{ \epsilon_t \} $ satisfies an invariance principle, i.e.
\[ T^{-1/2} \sum_{i=1}^{\trunc{Ts}} \epsilon_i \Rightarrow \sigma B(s), \] as $ T \to \infty $, in $ D([0,1];\R) $, for some constant $ \sigma \in (0,\infty) $
and Brownian motion $B$, then 
\begin{equation}
\label{FCLT ST}
  S_T \Rightarrow B_\delta^\sigma = \int \delta \, d \lambda + \sigma B,
\end{equation}
as $ T \to \infty $, in $ D([0,1];\R) $.
\end{lemma}

\begin{proof}  Notice that  $ S_T $ attains the decomposition $ S_T(u) = T^{-1/2} \sum_{i=1}^{\trunc{Tu}} \epsilon_i + A_T(u) $, 
where the first term is a martingale and $A_T(u) = T^{-1/2} \sum_{i=1}^{\trunc{Tu}} E(Y_i) = T^{-1} \sum_{i=1}^{\trunc{Tu}} \delta(i/T) $ is non-random. But due to Koksma's theorem, 
\[
  \left| A_T(u) - \int_0^u \delta(z) \, dz \right| \le V( \delta ) T^{-1},
\]
where the upper bound is independent from $u$.  The variation of the step function $ A_T(u) $ is $ T^{-1} \sum_{i=1}^T |\delta(i/T)| $, which
converges if $ \delta $ is piecewise Lipschitz with a finite number of finite jumps, and is therefore bounded in $T \ge 1 $. Hence, $ S_T $ is a semimartingale.
\end{proof}

We shall impose mixing conditions on the innovation process $ \{ \epsilon_t : t \in \Z \} $, which is assumed to be indexed by 
the integers. Recall that $ \{ \epsilon_t \} $ is called
$ \alpha $-mixing, if $ \alpha(k) = o(1) $, as $ k \to \infty $, where for $ k \in \N_0 $
\[
  \alpha(k) = \sup_{A \in \calF_{-\infty}^0, B \in \calF_{k}^\infty} | P(A \cap B ) - P(A) P(B) |
\]
denotes the $ \alpha $-mixing coefficient and $ \calF_a^b = \sigma( \epsilon_t : a \le t \le b ) $ for $ -\infty \le a \le b \le \infty $. 
$ \{ \epsilon_t \} $ is called $ \phi $-mixing, if $ \phi(k) = o(1) $, as $ k \to \infty $, where
\[
  \phi(k) = \sup_{A \in \calF_{-\infty}^0, B \in \calF_{k}^\infty} | P(A | B ) - P(A) |
\]
One can check that $ \alpha(k) \le \phi(k) $, see \citet{Doukhan1994} or \citet{AthreyaLahiri2006}.

\section{ASYMPTOTIC THEORY FOR SEQUENTIAL CROSS-VALIDATION}
\label{Sec:Asymptotics}

We shall now study the weak convergence theory of the sequential cross-validation bandwidth procedure.

Let us first identify the random processes which we have to investigate. Notice that for any $ s \in [s_0,1] $ and $ h > 0 $ we have
\[
  CV_s(h) = \frac{1}{T} \sum_{i=\trunc{Ts_0}}^{\trunc{Ts}} Y_i^2 - \frac{2}{T} \sum_{i=\trunc{Ts_0}}^{\trunc{Ts}} Y_i \widehat{m}_{h,-i}
    + \frac{1}{T} \sum_{i=\trunc{Ts_0}}^{\trunc{Ts}} \widehat{m}_{h,-i}^2,
\]
such that minimizing $ CV_s(h) $ is equivalent to minimizing the random function
\[
  C_{T,s}(h) = L_T(s) + Q_T(s),
\]
on which we shall focus in the sequel. Here the \cadlag processes $ \{ L_T(s) : s \in [s_0,1] \} $ and $ \{ Q_T(s) : s \in [s_0,1] \} $ are defined by 
\begin{align*}
  L_T(s)  & =  - \frac{2}{T} \sum_{i=\trunc{Ts_0}}^{\trunc{Ts}} Y_i \widehat{m}_{h,-i}, \\
  Q_T(s) &=   \frac{1}{T} \sum_{i=\trunc{Ts_0}}^{\trunc{Ts}} \widehat{m}_{h,-i}^2,
\end{align*}
for $ s \in [s_0,1] $; for our study it will be convenient to omit the $h$ in the notation. 

We shall see that $ L_T $ and $ Q_T $ have different convergence rates, $ Q_T $ being the leading term which determines the
asymptotics of $ C_{T,s}(h) $ for large $ T $. After scaling appropriately their weak limits turn out to be functionals of the
process
\begin{equation}
\label{semimart}
  B_\delta^\sigma = \int \delta \, d \lambda + \sigma B
\end{equation}
which appears as the limit of the partial sum process of the observations, $ Y_n = Y_{Tn} $, confer
Lemma~\ref{FCTLPartialSum}. Recall that $ m_0 + \delta/\sqrt{T} $ is the regression function after the (first) change-point.
Thus, the limit theorems show the effect of a general departure from the no-change model $ m_0 $ given by the function
$ \delta $, which appears as the drift in the semimartingale (\ref{semimart}).

As already mentioned in the previous section, we shall impose weak conditions on the $ \alpha $-mixing coefficients of the
innovation process $ \{ \epsilon_t \} $ of martingale differences. Indeed, those conditions are naturally satisfied by
many time series studied in the literature. As an example, consider the $ GARCH(p,q) $ model given by
\[
  \epsilon_t = \sigma_t \xi_t, \qquad \sigma^2_t = \alpha_0 + \sum_{j=1}^p \alpha_j \epsilon_{t-j}^2
   + \sum_{j=1}^q \beta_j \sigma_{t-j}^2,
\]
where $ \{ \xi_t \} $ are i.i.d.$(0,1)$ random variables, $ \alpha_p \beta_q \not= 0 $, $ \alpha > 0 $ and
$ \alpha_i, \beta_j \ge 0 $ for $ i = 1, \dots, p $ and $ j = 1, \dots, q $. It is known that a strictly stationary
GARCH$(p,q)$ process is $ \phi $-mixing with geometrically decreasing $ \phi $-mixing coefficients,
if $ \xi_1 $ attains a Lebesgue density, cf. \citet{Doukhan1994}. This implies geometrically decreasing 
$ \alpha $-mixing coefficients, which in turn implies that the conditions imposed in the results of the
present section on the $ \alpha $-mixing coefficients are  satisfied.

\subsection{The Process $ Q_T $}

Let us start our theoretical investigation with the more involved process $ Q_T $.

\begin{theorem} 
\label{ThQT}
\begin{itemize}
\item[(i)] Suppose that $ \{ \epsilon_n \} $ is a mean zero stationary martingale difference sequence which satisfies an invariance principle.
  Then, the fidi convergence of the process $ T^2 Q_T $ is given by
\[
  T^2( Q_T(s_1), \dots, Q_T(s_N) ) \Rightarrow  \diag \left(  \int_\gamma^{s_i} G^N(v) \, d B_\delta^\sigma(v)  \right)_{i=1}^N
\]
as $ T \to \infty $, where $  \int G^N(v) \, d B_\delta^\sigma(v) $ is the process 
\[\left \{  \int_\gamma^s \eins_{\{s \le s_i\}} \left( \int_\gamma^{s_i} g^{v,s_i}(u) \, d B_\delta^\sigma(u) \right)_{i=1}^N \, d B_\delta^\sigma(v) :
   s \ge \gamma \right\},
\]
for fixed time instants $ s_0 \le s_1 < \cdots < s_N \le 1 $.
Here 
\begin{equation}
\label{DefB}
  B_\delta^\sigma(u) = \int_0^u \delta(t) \, dt + \sigma B(u),
  \qquad u \in [0,1],
\end{equation}
and
\begin{equation}
\label{DefG}
  g^{v,s}(u) = \int_\gamma^s D( \xi u, \xi v, \xi w) N^{-2}( w ) \, dw, 
  \qquad u,v \in [0,w], s \in [s_0,1], 
\end{equation}
where 
\begin{align}
\label{DefD}
  D(u,v,w) &= 
  \left\lbrace 
  \begin{array}{ll}
     K(w-u) K(w-v), \qquad  & u,v, w \in [0,\infty),\ u,v \le w, \\
     0, & \text{otherwise},
   \end{array}
   \right.  \\
\label{DefN}
   N(w) & = \int_\gamma^{w}  K(  \xi (w - z) ) \, dz, \qquad w \in [\gamma,1] .
\end{align}
\item[(ii)]   Let $ \{ \epsilon_n \} $ be a strictly stationary martingale difference sequence with $ E(\epsilon_1) = 0 $, $ E(\epsilon_1^8) < \infty $ for some $ \delta > 0 $ and $\alpha $-mixing coefficients, $ \alpha(k) $, satisfying
  \[
   \sum_{k=0}^\infty [\alpha(k)]^{3/4}  < \infty \quad \text{and} \quad \sum_{k=1}^\infty k^{1+\zeta} [\alpha(k)]^{1-\zeta} < \infty,
  \]
  for some $ \zeta \in (0,1) $. Then the process $ \{ Q_T(s) : s \in [s_0,1] \} $ is tight and therefore converges weakly.
\end{itemize}
\end{theorem}

\begin{proof} Denote by $ S_T $ the partial sum process introduced in Lemma \ref{FCTLPartialSum}.
Either by the assumption stated in (i) or under the  moment and mixing conditions imposed in (ii), we have the weak convergence
\[
  S_T(u) \Rightarrow B_\delta^\sigma(u) = \int_0^u \delta(t) \, dt + \sigma B(u), 
\]
as $ T \to \infty $, since we may apply \citet[Corollary 1]{Herrndorf1984} with $ \beta = 4 $ under condition (ii). Indeed, the conditions on the mixing coefficients are
stronger than required there and $ E(\sum_{i=1}^n \epsilon_i)^2/n = E\epsilon_1^2 < \infty $ holds true for any strictly stationary martingale difference sequence
$ \{ \epsilon_t \} $.
We shall now apply the Skorohod representation theorem which asserts that on a new probability space 
equivalent versions of the processes $ \{ S_T(u) : u \in [s_0,1] \} $ and $ \{ B_\delta^\sigma(u) : u \in [s_0,1] \} $ can be defined,
which we will again denote by $ S_T $ and $ B_\delta^\sigma$, such that
\[
  \| S_T - B_\delta^\sigma \|_\infty \to 0, \qquad a.s.,
\]
as $ T \to \infty $. Let us consider the quadratic form $ Q_T(s) $. Notice that
\begin{align*}
  Q_T(s) & = \frac{1}{T^3} \sum_{i=\trunc{Ts_0}}^{\trunc{Ts}} 
  \frac{ \sum_{j,k=\trunc{T\gamma}, j \not= k}^{i-1} K(i/h-j/h)K(i/h-k/h) Y_j Y_k }
  { \left( \frac{1}{T} \sum_{j=\trunc{T\gamma}}^{i-1} K( [i-j]/h ) \right)^2 } \\
  & =\frac{1}{T^2} \sum_{i=\trunc{Ts_0}}^{\trunc{Ts}}  \sum_{j,k=\trunc{T\gamma}}^{\trunc{Ts}} 
      \frac{ D(j/h, k/h, i/h ) } { \left( \frac{1}{T} \sum_{j'=1}^{i-1} K( [i-j']/h ) \right)^2  } \frac{Y_j}{\sqrt{T}} \frac{Y_k}{\sqrt{T}}, 
\end{align*}
where the function $ D : [0,\infty)^3 \to [0,\infty) $ is defined in (\ref{DefD}).
Using the fact that
\[
  \frac{1}{T} \sum_{i=\trunc{Ts_0}}^{\trunc{Ts}-1} K( i/h - j/h ) K( i/h - k/h ) 
    = \int_{\trunc{Ts_0}/T}^{\trunc{Ts}/T-1/T}  K( \trunc{Tx}/h - j/h ) K( \trunc{Tx}/h - k/h ) \, dx,
\]
we may represent $ T^2 Q_T(s) $ via (It\^o) integrals, namely
\begin{align*}
 & T^2 Q_T(s) \\
 & \qquad =  \int_{\trunc{Ts_0}/T}^{\trunc{Ts}/T} \int_{\trunc{T\gamma}/T}^{\trunc{Ts}/T} \int_{\trunc{T\gamma}/T}^{\trunc{Ts}/T} D( uT/h, vT/h,\trunc{Tw}/h ) N_T^{-2}( w )  \, d S_T(u) \, d S_T(v) \, dw \\
   & \qquad =  
   \int_{\trunc{T\gamma}/T}^{\trunc{Ts}/T} \int_{\trunc{T\gamma}/T}^{\trunc{Ts}/T} \int_{\trunc{Ts_0}/T}^{\trunc{Ts}/T} D( uT/h, vT/h,\trunc{Tw}/h) N_T^{-2}( w )  \, dw \, d S_T(u) \, d S_T(v),
\end{align*}
where
\begin{equation}
\label{DefNT}
  N_T( w ) = \frac{1}{T} \sum_{j'=\trunc{T\gamma}}^{ \trunc{Tw}-1} K( \trunc{Tw}/h - j'/h )
  = \int_{\trunc{T\gamma}/T}^{ \trunc{Tw}/T-1/T}  K( \trunc{Tw}/h - \trunc{Tz}/h ) \, dz, \qquad w \ge 0.
\end{equation}
The first step will be to apply Theorem \ref{KurtzProtter} to obtain weak convergence of the inner It\^o integral. The second step, a diagonal argument, will
then yield the fidi convergence. Lastly, we verify tightness under the conditions given in (ii).
Clearly, we expect that $ N_T(w) $ converges to the function $ N(w) = \int_\gamma^{w}  K(  \xi (w - z) ) \, dz $, $ w \in [s_0,1]$. 
Since for $ w \ge \gamma $ we have $ N(w) \ge N(\gamma) > 0 $ and, of course,
\begin{equation}
\label{UnifKonvN}
  \sup_{ w \in [\trunc{T\gamma},1]} | N_T(w) - N(w) | \le V(K) T^{-1}
\end{equation}
by virtue of Koksma's theorem, yielding $ | N_T^{-1}(w) - N^{-1}(w) | \to 0 $, as $ T \to \infty $, uniformly in $ w \in [\gamma,1] $.
Fix $ s \ge s_0 $ and $ v $. Define for $ u \in [0,1] $
\begin{align*}
  g_T^{v,s}(u) & = \int_{\trunc{T\gamma}/T}^{\trunc{Ts}/T} D( u T/h, v T/h,\trunc{Tw}/h ) N_T^{-2}( w ) \, dw, \\
  g^{v,s}(u) & = \int_\gamma^s D( \xi u, \xi v, \xi w) N^{-2}( w ) \, dw.
\end{align*}
For what follows, we need to verify that $ g_T^{v,s} \to g^{v,s} $ in the uniform topology, and that $ g_T^{v,s} $ has
uniformly bounded variation. Clearly, $ | g_T^{v,s}(u) - g^{v,s}(u) | $ can be bounded by
\[
  O(T^{-1}) + 
  \int_{\trunc{T\gamma}/T}^{\trunc{Ts}/T} | D( u T/h, v T/h, \trunc{Tw}/h ) N_T^{-2}( w ) - D(\xi u,\xi v,\xi w)N^{-2}( w ) | \, dw.
\]
Let $ A_T = \{ (u,v,w) : \trunc{T\gamma} \le u,v, \trunc{Tw}/h \le \trunc{Ts}/T  \} $. On the set $ A_T $ the above integrand equals
$ | K( \trunc{Tw}/h - u T/h ) K( \trunc{Tw}/h - v T/h )  N_T^{-2}(w) - K(\xi(w-u)) K(\xi(w-v)) N^{-2}(w) | $. Recall the fact that
for sequences of mappings $ \{ a, a_T \} $, $ \{ b, b_T\} $ taking values in some normed space with norm $ \| \cdot \| $, 
we have $ a_T b_T \to fg $, as $ T \to \infty $, provided $ a_T \to a, b_T \to b $ and $ \| a \|, \sup_{T \ge 1} \| b_T \| < \infty $. 
Apply that result with $ a_T(u,v,w) = K( \trunc{Tw}/h - u T/h ) K( \trunc{Tw}/h - v T/h ) $, $ a(u,v,w) = K( \xi(w-u) ) K( \xi(w-v) ) $, 
$ b_T(w) = N^{-2}(w) $ and $ b(w) = N^{-2}(w) $.
By boundedness and Lipschitz continuity of $K$ and due to  (\ref{UnifKonvN}) we may conclude that 
\begin{equation}
\label{UnifConvGT}
  \sup_{u,v \in [\gamma,1], w \ge s_0} | g_T^{v,s}(u) - g^{v,s}(u) | \to 0,
\end{equation}
as $ T \to \infty $; that convergence is even uniform in $ s \in [s_0,1] $. Before proceeding, let us check that $ g_T^{v,s} $ is of uniformly bounded
variation, such that the uniform limit $ g^{v,s} $ is of bounded variation as well. Clearly, $ g_T^{s,v} $ is a step function with
jumps at $ k/T $, $ k = \trunc{T \gamma}/T, \dots, \trunc{Ts}-1 $, of size not larger than $ T^{-1} \| K \|_\infty^2 / N(\gamma)^2 $ in absolute value. Thus, for any
partition $ \{ \xi_i \} $, arbitrary $ s \in [s_0,1] $ and $ v \le w $, the variation $ \sum_i \left| g_T^{v,s}(\xi_{i+1}) - g_T^{v,s}( \xi_i ) \right| $ can be
bounded by $ \| K \|_\infty^2 /N(\gamma)^2 $, yielding
\begin{equation}
\label{UnifTVGT}
\sup_{s\in[s_0,1], \gamma \le v \le w} \sup_{T \ge 1} V( g_T^{v,s} ) < \infty.
\end{equation}
By (\ref{UnifConvGT}), we may conclude (take $ \lambda = \text{id} $) that, for fixed $ v,s $,
\begin{equation}
\label{SM0}
  \inf_{\lambda \in \Lambda} \max \left \{
    \left\| \sqrt{   (S_T \circ \lambda( \cdot ) - B_\delta^\sigma \circ \lambda( \cdot )  )^2 
                       + (g_T^{v,s} \circ \lambda( \cdot ) - g^{v,s} \circ \lambda(\cdot) )^2 } \right\|_\infty ,
    \| \lambda - \id \|_\infty 
    \right\} = o( 1 ),
\end{equation}
as $ T \to \infty $, a.s., where the $ o(1) $ is even uniform in $ u,v $ and $ \| \cdot \|_\infty $ denotes the supnorm over $[\gamma,1] $.  This means, $ d((g_T^{v,s}, S_T ),(g^{v,s},B_\delta^\sigma) ) \to 0 $,
as $ T \to \infty $, a.s.,  which, of course, implies weak convergence by virtue of the second half of the Skorohod/Dudly/Wichura
representation theorem, i.e.
\[
  ( g_T^{v,s}, S_T ) \Rightarrow ( g^{v,s}, B_\delta^\sigma ),
\]
as $ T \to \infty $, in the Skorohod space $ D( [\gamma,1]; \R^2 ) $. We may apply Theorem \ref{KurtzProtter} to conclude that
\[
  \left( g_T^{v,s}, S_T, W_T^{v,s}\right) \Rightarrow \left( g^{v,s},  B_\delta^\sigma, W^{v,s} \right),
\]
in $ D([\gamma,1]; \R^3) $, as $ T \to \infty $, for the equivalent versions, where the processes $ \{ W^{v,s}(t) : t \in [\gamma,1] \}, \{ W_T^{v,s}(t) : t \in [\gamma,1] \} $, $ T \ge 1 $,
are defined by
\begin{align*}
  W_T^{v,s} & =  \int g_T^{v,s}(u) \, d S_T(u), \\
  W^{v,s} & =  \int g^{v,s}(u) \, d B_\delta^\sigma(u).
\end{align*}
The second step is a {\em diagonal argument}: Fix $ N \in \N $ and points $ s_1, \dots, s_N \in [s_0,1] $ with $ s_1 < \dots < s_N $. Put for $ T \ge 1 $
\begin{align*}
  G_T^N(v) &= \left( \eins_{\{ v \le s_1 \} } \int_{\trunc{T\gamma}/T}^{s_1} g_T^{v,s_1}(u) \, d S_T(u), \dots, \eins_{ \{ v \le s_N \} }  \int_{\trunc{T\gamma}/T}^{s_N} g_T^{v,s_N}(u) \, d S_T(u) \right), \\
  G^N(v) & = \left( \eins_{\{ v \le s_1 \} ‚}  \int_{\gamma}^{s_1} g^{v,s_1}(u) \, d B_\delta^\sigma(u), \dots, \eins_{\{ v \le s_N \} }  \int_{\gamma}^{s_N} g^{v,s_N}(u) \, d B_\delta^\sigma(u)  \right).
\end{align*}
Let us check that  $ d( G_T, G ) = o(1) $, as $ T \to \infty $, where $d$ denotes the Skorohod metric on $ D([s_0,1];\R^N ) $. Consider for $ i =1, \dots, N$,
\[
  \int_{\trunc{T\gamma}/T}^{s_i} g_T^{v,s_i} \, d S_T - \int_{\gamma}^{s_i} g^{v,s_i} \, d B_\delta^\sigma 
  = O(T^{-1}) +  \int_\gamma^{s_i} (g_T^{v,s_i} - g^{v,s_i} ) \, d B_\delta^\sigma
  + \int_{\gamma}^{s_i} g_T^{v,s_i} d ( S_T - B_\delta^\sigma ).  
\]
The first integral on the right side converges in probability to $ 0 $, as $ T \to \infty $, since our assumptions on $ \delta $ ensure that
$ B_\delta^\sigma $ is a semimartingale. The second integral can be interpreted as a stochastic Stieltjes integral, since the integrand is of (uniformly) bounded variation.
Using integration by parts, (\ref{UnifConvGT}) and (\ref{UnifTVGT}), we see that, with $ \| \cdot \|_\infty $ denoting  the supnorm over $ [s_0,1] $,
\[
  \left| \int_\gamma^{s_i} g_T^{v,s_i} d ( S_T - B_\delta^\sigma ) \right| \le 2 \sup_{v,s} \| g_T^{v,s} \|_\infty \| S_T - B_\delta^\sigma \|_\infty 
  + \| S_T - B_\delta^\sigma \|_\infty \sup_{s \in [s_0,1], v \le w} V( g_T^{v,s} ) ,
\]
but the right side converges to $0$, as $ T \to \infty $, a.s. Now $ d(G_T,G) = o(1) $ follows easily.
A further application of Theorem \ref{KurtzProtter} yields
\[
  \left( G_T^N, S_T, \int G_T^N \, d S_T \right) \Rightarrow \left( G^N,  B_\delta^\sigma,  \int G^N \, d B_\delta^\sigma \right),
\]
as $ T \to \infty $, in $ D([s_0,1];\R^3) $, where by definition $  \int G^N \, d B_\delta^\sigma $ is the process
\[
 \left\{  \int_0^s \eins_{\{v \le s_i\}} \left( \int_\gamma^{s_i} g^{v,s_i}(u) \, d B_\delta^\sigma(u) \right)_{i=1}^N \, d B_\delta^\sigma(v) :
   s \in [s_0,1] \right\}.
\]
Now we sample the process $ \int G^N \, d B_\delta^\sigma $ at the points $ s_1, \dots, s_N $.
Then the diagonal of the $ N \times N $ matrix with $i$th row given by the vector
\[ 
  \int_0^{s_i} G^N(v) \, d B_\delta^\sigma(v) ,
\]
$ i = 1, \dots, N $, equals $ ( Q_T(s_1), \dots, Q_T(s_N) ) $. Consequently, we may conclude that
\[ 
   T^2( Q_T(s_1), \dots, Q_T(s_N) )  \Rightarrow \diag \left(  \int_\gamma^{s_i} G^N(v) \, d B_\delta^\sigma(v)  \right)_{i=1}^N,
\]
as $ T \to \infty $, which completes the proof of (i). Let us now verify that under the assumptions  given in (ii) tightness of $ T^2 Q_T $ follows. 
Let $ s_0 \le a < b \le 1 $ and notice that due to (\ref{ModelPaper}) for $ \trunc{Ta} \le i \le \trunc{Tb} $
\begin{align*}
  E( \wh{m}_{h,-i} ) 
  & =
  \frac{1}{\sqrt{T}} \left[
    \sum_{j=\trunc{T\gamma}}^{i-1} K((i-j)/h) E(Y_j) / \sum_{j=1}^{i-1} K((i-j)/h) 
  \right] \\
  &  \le \frac{1}{\sqrt{T}} \left[
      \frac{ \sup_{s \in [a,b]} \int_\gamma^s K(\xi(s-z)) \delta(z) \, dz + o(1) }
      { \inf_{s\in[a,b]} \int_\gamma^s K(\xi(s-z)) \, dz + o(1) } 
    \right] \\
    & = O(1/\sqrt{T}),
\end{align*}
by positivity of the kernel, where the $ o(1) $ terms are uniform in $s$ and $i$, by virtue of Koksma's theorem.  We have 
\[
 E ( T^2 [ Q_T(b) - Q_T(a) ] )^4 
  = T^4 \sum_{i_1, \dots, i_4=\trunc{Ta}}^{\trunc{Tb}} E \left( \prod_{j=1}^4 \widehat{m}_{h,-i_j}^2 \right).
\]
When writing $ \wh{m}_{h,-i_j}^2 = ([\wh{m}_{h,-i_j} - E\wh{m}_{h,-i_j}] + E\wh{m}_{h,-i_j})^2 $, multiplying out and collecting terms,
we see that only the terms involving 
\[ 
  \wh{m}_{h,-i_j} - E\wh{m}_{h,-i_j} = \sum_{l=\trunc{T\gamma}}^{i_j-1} K((i_j-l)/h) \left(\sum_{l'=\trunc{T\gamma}}^{i_j-1} K((i_j-l')/h) \right)^{-1} (Y_l - E(Y_l)) 
\]
but not
$ E \wh{m}_{h,-i_j} $ have to be dealt with, since for $ \rho = 1, \dots, 8 $
\[
  E [\wh{m}_{h,-i_j} - E\wh{m}_{h,-i_j}]^{8-\rho}( E \wh{m}_{h,-i_j} )^{\rho} = O( T^{-\rho/2} ). 
\]
Therefore we can and will assume from now on that $ E(Y_j) = 0 $. 
For $ \trunc{Ta} \le i_1, \dots, i_4 \le \trunc{Tb} $ we have by non-negativity of $K$ and since
$ N_T( i_\nu / T ) \ge O(1/T) $ for $ \nu = 1, \dots, 4 $,
\begin{align*} 
  & E\left( \prod_{j=1}^4 \widehat{m}_{h,-i_j}^2 \right)  \\
  & \quad =  \sum_{j_1,k_1= \trunc{T\gamma}}^{i_1-1} \cdots \sum_{j_4,k_4=\trunc{T\gamma}}^{i_4-1}
  \frac{ \prod_{\nu = 1}^4 \prod_{l \in \{ j_\nu, k_\nu \}}  K(i_\nu/h - l/h)  }{ \prod_{\nu=1}^4 N_T(i_\nu/T)^2  } E(Y_{j_1} Y_{k_1} \cdots  Y_{j_4} Y_{k_4} )
   \\
  & \quad= O\left( \frac{\| K \|_\infty^8}{T^8} \sum_{j_1,k_1=\trunc{T\gamma}}^{i_1-1} \cdots \sum_{j_4,k_4=\trunc{T\gamma}}^{i_4-1} | E(Y_{j_1} Y_{k_1} \cdots  Y_{j_4} Y_{k_4}) | \right)\\
  & \quad = O( \max( i_1,\dots,i_4' )^4 / T^8 ) 
\end{align*}
Here we used the fact that a strictly stationary sequence $ \{ \xi_n \} $ ensuring the imposed moment and $ \alpha $-mixing conditions
satisfies $ \sum_{i_1,\dots,i_{2m}=1}^n | E(\xi_{i_1} \cdots \xi_{i_{2m}} ) | = O(n^m ) $, for $ m \in \N $, cf. \citet[proof of Theorem 1, p. 47]{Yokoyama1980}
and \citet{Kim1993} for the slightly weaker conditions.
Thus,
\[
 E ( T^2 [ Q_T(b) - Q_T(a) ] )^4 = O( |b-a|^4 )
\]
H\"older's inequality now ensures that for $ s_0 \le s_1 \le s_2 \le 1 $
\begin{align*}
 & E|T^2Q_T(s)-T^2Q_T(s_1)|^2|T^2Q_T(s_2)-T^2Q_T(s)|^2 \\
  & \qquad \le \sqrt{ E | T^2 Q_T(s) - T^2 Q_T(s_1) |^4 } \sqrt{ E | T^2Q_T(s_2) - T^2 Q_T(s) |^4 } \\
  & \qquad = O( | s - s_1 | | s_2 - s | ) \\
  & \qquad = O( |s_2 - s_1|^2 ),
\end{align*}
which verifies the criterion \citet[Theorem 15.6]{Billingsley1968}.
\end{proof}

\subsection{The Process $ L_T $ and the Cross-Validation Criterion}

The next theorem provides a functional central limit theorem for the process $ L_T $.

\begin{theorem}
\label{ThLT}
  Let $ \{ \epsilon_n \} $ is a strictly stationary sequence with $ E(\epsilon_1) = 0 $, $ E(\epsilon_1^4) < \infty $ and  $\alpha $-mixing coefficients,
  $ \{ \alpha(k) \} $, satisfying
  \[
   \sum_{k=0}^\infty [\alpha(k)]^{1/2}  < \infty \quad \text{and} \quad \sum_{k=1}^\infty k^{1+\zeta} [\alpha(k)]^{1-\zeta} < \infty,
  \]
  for some $ \zeta \in (0,1) $.
  Then
  \begin{equation}
  \label{DefL}
    T L_T(s) \Rightarrow \calL_\xi(s) = -2  \int_0^s 
     \frac{ \int_0^u K( \xi(u-v) ) \, d B_\delta^\sigma(v) }{ \int_0^u K( \xi(u-v) ) \, dv } \, d B_\delta^\sigma(u),
  \end{equation}
  in $ D([0,1];\R) $, as $ T \to \infty $. $ \calL_\xi $ is a.s. continuous.
\end{theorem}

\begin{proof}
Again, by virtue of the Skorohod/Dudley/Wichura representation theorem, we assume w.l.o.g. that $ \| S_T -  B_\delta^\sigma \|_\infty \to 0 $, a.s., as $ T \to \infty $. Notice that
\begin{align*}
  L_T(s) & = - \frac2T
    \sum_{i=\trunc{Ts_0}}^{\trunc{Ts}} 
      \frac{ \sum_{j=\trunc{T\gamma}}^{i-1} K((i-j)/h) Y_j / \sqrt{T} }{ T^{-1} \sum_{j'=1}^{i-1} K((i-j')/h) } \frac{Y_i}{\sqrt{T}} \\
      & = - \frac2T \int_{\trunc{Ts_0}/T}^{\trunc{Ts}/T}
      \left( u \mapsto 
      \frac{ \sum_{j=1}^{\trunc{Tu}-1} K((\trunc{Tu}-j)/h) Y_j/\sqrt{T} }{ T^{-1} \sum_{j'=1}^{\trunc{Tu}-1} K((\trunc{Tu}-j')/h) }
      \right) \, d S_T(u)
\end{align*}
leading us to the representation
\[
  T  L_T(s) = -2 \int_{\trunc{Ts_0}/T}^{\trunc{Ts}/T} I_T(u) \, d S_T(u)
\]
with
\begin{align*}
  I_T(u) &=
    \int_{\trunc{T\gamma}/T}^{\trunc{Tu}/T-1/T} 
      K( ( \trunc{Tu} - \trunc{Tv} )/h ) \, d S_T(v) \ \biggl/ \ 
      \int_{\trunc{T\gamma}/T}^{\trunc{Tu}/T-1/T} K( ( \trunc{Tu} - \trunc{Tz} )/h ) \, dz \\
      & = \int_{\trunc{T\gamma}/T}^{\trunc{Tu}/T-1/T}  E_T^u(v) \, d S_T(v), 
\end{align*}
where $ N_T $  is defined in (\ref{DefNT}) and
\[
  E_T^u(v) = K( ( \trunc{Tu} - \trunc{Tv} )/h ) N_T^{-1}(u), \qquad u,v \in [s_0,1], v \le u.
\]
Recall that $ N_T(s) $, $ s \in [s_0,1] $, is not smaller than $ \inf_{s \in [s_0,1]} \int_0^{s} K(\xi(s-z)) \, dz + o(1) $ which is bounded away from $0$.
As in the proof of Theorem~\ref{ThQT}, one can show that for fixed $ u $
\[
  \left( v \mapsto K((\trunc{Tu}-\trunc{Tv})/h) N_T^{-1}(u), S_T \right)
  \Rightarrow
  ( v \mapsto K(\xi(u-v)), S_T ),
\]
in $ D( [s_0,u]; \R^2 ) $, as $ T \to \infty $, such that Theorem \ref{KurtzProtter} guarantees that the process
\[
  \left( v \mapsto K((\trunc{Tu}-\trunc{Tv})/h), S_T,   \left\{ \int_{1/T}^{\trunc{Tu'}/T - 1/T} K( ( \trunc{Tv} - \trunc{Tu} )/ h ) \, d S_T(v) : u' \in [s_0,u] \right\} \right)
 \]
converges weakly in $ D([s_0,u];\R^3) $ to  the process
\[
   \left( v \mapsto K( \xi(u-v) ), S_T,  \left\{ \int_0^{u'} K( \xi(u-v) ) \, d B_\delta^\sigma(v) : u' \in [s_0,u]
 \right\} \right).
\]
Now we apply the diagonal argument given in the proof of  Theorem \ref{ThQT} to obtain the fidi
convergence of  $ I_T $,
\[
  I_T(u) \stackrel{fidi}{\to}  I(u) = \int_0^u K( \xi(u-v) ) \, d B_\delta^\sigma(u) \ \biggr / \ \int_0^u K( \xi(u-v) ) \, dv, 
\]
as $ T \to \infty $. To extend that result to weak convergence in $ D([0,1];\R) $, it remains to show tightness of the process $ L_T $. 
We may argue as in the proof of Theorem \ref{ThQT}. Again applying \citet[proof of Theorem 1 p. 47]{Yokoyama1980}, we obtain
\begin{align*}
  E\left( \int_{\trunc{Ta}}^{\trunc{Tb}} E_T(u) \, d S_T(u) \right)^4
  & = \frac{1}{T^2} \sum_{i_1,\dots,i_4 = \trunc{Ta}}^{\trunc{Tb}} \prod_{j=1}^4 E_T(i_j/T) E(Y_{i_1} \cdots Y_{i_4} ) \\
  & \le \sup_{x \in \R, T \ge 1} | E_T(x) |^2 \frac{1}{T^2} \sum_{i_1,\dots,i_4=\trunc{Ta}}^{\trunc{Tb}} | E(Y_{i_1} \cdots Y_{i_4}) | \\
  & = O\left( \sup_{x \in \R, T \ge 1} | E_T(x) |^2 \left( \frac{ \trunc{Tb} - \trunc{Ta} }{T} \right)^2 \right) \\
  & = O( |b-a|^2 ).
\end{align*}
Thus, for $ s_0 \le r \le s \le 1 $,
\[
  \| I_T(s) - I_T(r) \|_4 = 
  \left\| -2 \int_{\trunc{Tr}/T}^{\trunc{Ts}/T} E_T(u) \, d S_T(u) \right\|_4
  = O( |s-r|^{1/2} ).
\]
H\"older's inequality now entails that
\begin{align*}
  E | I_T(s) - I_T(s_1) |^2 | I_T(s_2) - I_T(s) |^2
  & \le 
  \sqrt{ E | I_T(s) - I_T(s_1) |^4 }  \sqrt{ E | I_T(s_2) - I_T(s) |^4 } \\
  & = O( |s-s_1| | s_2 -s | ) \\
  & = O( | s_2 - s_1 |^2 ),
\end{align*}
thus establishing tightness. We can conclude that 
\[ I_T \Rightarrow I \qquad \text{in $ D([s_0,1];\R) $}, \] as $ T \to \infty $. Again considering equivalent processes
on a new probability space, we may assume that $ \| S_T - B_\delta^\sigma \|_\infty \to 0 $ as well
as $ \| I_T - I \|_\infty $, as $ T \to \infty $. The same argument as used to obtain (\ref{SM0}) yields 
$ (S_T, I_T) \Rightarrow (B_\delta^\sigma, I) $ in $ D([s_0,1],\R^2) $, as $ T \to \infty $. A further
application of Theorem \ref{KurtzProtter} yields
\[
  (S_T, I_T, T L_T) = \left( S_T, I_T, \int L_T \, d S_T \right) \Rightarrow \left( B_\delta^\sigma, I, \calL_\xi \right),
\]
in $ D([s_0,1];\R^3 ) $, as $ T \to \infty $, which completes the proof.
\end{proof}

We may now easily combine the results of Theorem \ref{ThQT} and Theorem \ref{ThLT}. Since the convergence rates
of $ Q_T $ and $ L_T $ differ, the asymptotic distribution of $ T C_{T,s} $ is dominated by the process $ T L_T $.

\begin{theorem} 
\label{FCLTQT}
 Suppose that $ \{ \epsilon_n \} $ is a strictly stationary martingale difference sequence with $ E(\epsilon_1) = 0 $, $ E(\epsilon_1^8) < \infty $ and
  $\alpha $-mixing coefficients satisfying
  \[
   \sum_{k=0}^\infty [\alpha(k)]^{3/4}  < \infty \quad \text{and} \quad \sum_{k=1}^\infty k^{1+\zeta} [\alpha(k)]^{1-\zeta} < \infty,
  \]
  for some $ \zeta \in (0,1) $. 
  Then the cross-validation objective function, $ C_{T,s}(h) $ satisfies a functional central limit theorem,
  \[
    T C_{T,s}(h) \Rightarrow \calL_\xi(s),  
  \]
  as $ T \to \infty $, in the space $ D([0,1];\R) $, where the process $ \calL_\xi $ is as in (\ref{DefL}).
\end{theorem} 

\subsection{The Cross-Validated Bandwidth Process}
\label{CrossValidatedPr}

To simplify the exposition, let us from now on strengthen Assumption \ref{ST:Bandwidth} to
\[
  h = h(\xi) = T/\xi,
\]
such that the problem is parameterized by $ \xi $. Let us assume that  optimization is done over a fine grid
\[
  \Xi = \{ \xi^1, \dots, \xi^M \},
\]
where $M \in \N$ is arbitrary large but fixed. Now at each time instant $ s $ the minimum
\[
  \xi_T^*(s) = \argmin_{\xi \in \Xi} \bar{C}_{T,s}( \xi ) 
\]
is calculated, where $ \bar{C}_{T,s}(\xi) = C_{T,s}( T/\xi ) $.  Here and in the sequel the operator $ \argmin_{a \in A} f(a) $ 
for a function $ f : A \to \R $ refers to the smallest $ a \in A $ such that $ f(a) \le f(x) $ for all $ x \in A $, thus leading
to an unique definition.

We obtain the following corollary.

\begin{corollary} Given the conditions of Theorem \ref{FCLTQT}, 
\begin{equation}
  \{ T \bar{C}_{T,\cdot}( \xi ) : \xi \in \Xi \} \Rightarrow \{ \calL_\xi(\cdot) : \xi \in \Xi \},
\end{equation}
as $ T \to \infty $, in the product space $ ( D( [s_0,1]; \R ) )^M $.  Consequently,
\[
  \xi_T^* \Rightarrow \argmin_{\xi \in \Xi} \calL_\xi,
\]
as $ T \to \infty $, in $D([s_0,1];\R) $.
\end{corollary}

\begin{proof}
  The process $ \{ T \bar{C}_{T,s}( \xi ) : \xi \in \Xi \}  $ is tight, since the coordinate processes 
  $ \{ T \bar{C}_{T,s}( \xi ) : s \in [s_0,1] \} $ are tight for each $ \xi \in \Xi $. To check convergence
  of the fidis, we consider a linear combination 
  \[ H_T(s) = \sum_{\xi \in \Xi} \lambda_\xi T^2 Q_T^\xi(s) \]
  for $ \lambda_\xi $, $ \xi \in \Xi $, such that not all $ \lambda_\xi $ vanish. We can represent $ H_T(s) $ as
  \[
   \int_{\trunc{Ts_0}/T}^{\trunc{Ts}/T} \int_{\trunc{Ts_0}/T}^{\trunc{Ts}/T} \int_{\trunc{T\gamma}/T}^{\trunc{Ts}/T} \sum_\xi D(uT/h, vT/h, wT/h) (N_T^\xi)^{-2}(w) \, dw \, d S_T(u) \, d S_T(v)
  \] 
  We have shown in the proof of Theorem \ref{FCLTQT} that for fixed $ \xi  \in \Xi $ 
  \[
    g_T^{v,s}(u;\xi) = \int_{\trunc{T\gamma}/T}^{\trunc{Ts}/T} D( u T/h, v T/h, \trunc{Tw}/h) (N_T^\xi)^{-2}(w) \, d w \, d S_T( u ) |, d S_T( v )
  \]
  converges uniformly in $u, v \in [\gamma,1] $ and $ w \ge s_0 $ to 
  \[
    g^{v,s}(u;\xi) = \int_\gamma^s D( \xi u, \xi v, \xi w )(N^\xi)^{-2}( w ) \, d w,
  \]
  as $ T \to \infty $. Then the triangle inequality shows that $ \sum_{\xi \in \Xi} \lambda_\xi  g_T^{v,s}(u;\xi) $ converges uniformly to
  $  \sum_{\xi \in \Xi} \lambda_\xi  g^{v,s}(u;\xi)  $, as $ T \to \infty $. Now we can apply exactly the same arguments as in the proof of Theorem \ref{FCLTQT} 
  to obtain the fidi convergence 
  \[ 
    ( H_T(s_1), \cdots, H_T(s_N) ) \Rightarrow \diag \left( \int_\gamma^{s_i} G^{N,\xi}(v) \, d B(v) \right)_{i=1}^N,
   \] 
   as $ T \to \infty $, for  fixed $ s_1, \dots, s_N $. The same chain of arguments shows that the fidis of $ \sum_{\xi \in \Xi} \lambda_\xi T L_T^\xi( \cdot ) $
  converge weakly to the fidis of $ \sum_{\xi \in \Xi} \lambda_\xi \calL_\xi(\cdot) $, such the fidi convergence of $ \sum_{\xi \in \Xi} \lambda_\xi \bar{C}_{T,\cdot}(\xi) $
  follows. Again, tightness of the linear combination follows easily from the triangle inequality for the $ L_p $ norm.
  Since $ \Xi $ is a finite set, we immediately obtain that $ \{ \bar{C}_{T,s}( \xi) : \xi \in \Xi \} $ converges weakly
  to $ \{ \calL_\xi : \xi \in \Xi \} $, as $ T \to \infty $. But this implies the weak convergence result for the smallest
  minimizer.
\end{proof}

\section{AN ANSCOMBE-TYPE THEOREM FOR RANDOM TIME HORIZONS}
\label{Sec:Anscombe}

The results of the previous sections assume that monitoring stops latest at the non-random time horizon $T$, and the 
theory is nicely captured by sequential empirical processes being elements of Skorohod spaces of functions defined on
$ [0,1] $, such as $ D([0,1];\R) $. Here the unit interval corresponds to the physical time interval $ [0,T] $. The limit
theorems then provide approximations to the true distribution of the sequential processes when $T$ is fixed but large.

Let us now assume that the time horizon $T$ is determined by a parameterized family of random experiments given by a family
$ \{ \tau_a : a > 0 \} $ of random variables, frequently stopping times, taking values in the natural numbers. This may happen, 
if, for example, the time horizon is determined as the time instant where cumulated costs exceed a threshold for the first time.
The question arises whether in limit theorems, say for (standardized) sums of $T$ terms, one may replace $ T $, assumed to tend to
$ \infty $, by a family of random variables indexed by $ a > 0 $, which behaves as $ \lambda  a $, $ \lambda $ a positive constant, 
as $ a \to \infty $, a condition which ensures that $ \tau_a $ tends to $ \infty $ as $ a \to \infty $, such that one can hope that
the asymptotics $ T \to \infty $ can be replaced by $ a \to \infty $ when replacing $ T $ by $ \tau_a $. This issue
has been extensively studied in the literature. Anscombe's seminal paper on this topic, \citet{Anscombe1952},
gave sufficient conditions for this to be true. Applied to sums of i.i.d. random variables, his result is as follows.

\begin{theorem} {\sc (Anscombe, 1952)}\\
\label{ThAnscombe}
Let $ X_1, X_2, \dots $ be i.i.d. random variables with mean $0$ and common variance $ \sigma^2 \in (0,\infty) $ and put
$ S_n = \sum_{i=1}^n X_i $, $ n \in \N $. Suppose that
the family $ \{ \tau_a : a > 0 \} $ of random indices satisfies 
\begin{equation}
\label{AnscCond}
  \frac{\tau_a}{a} \stackrel{P}{\to} \lambda \in (0, \infty), 
\end{equation}
as $ a \to \infty $. Then
\[
  \frac{ S_{\tau_a} }{ \sigma \sqrt{\tau_a} } \stackrel{d}{\to} N(0,1),
\]
as well as
\[
  \frac{ S_{\tau_a} }{ \sigma \sqrt{ \lambda a } }  \stackrel{d}{\to} N(0,1),
\]
as $ a \to \infty $.
\end{theorem}

Anscombe's result belongs to the fundamental insights on sequential methodologies and can be found in
various monographs such as \citet{Siegmund1985} or \cite{GoshEtAl1997}. It is worth mentioning that in its basic form it addresses
a sequence $ \{ Z, Z_n \} $ which converges weakly, i.e. $ Z_n \stackrel{d}{\to} Z $, as $ n \to \infty $. Provided
that given $ \varepsilon > 0 $ there exists $ \delta > 0 $ and $ n_0 \in \N $, such that
\begin{equation}
\label{UniformContinuity}
  P \left( \max_{\{ k : |k-n| < n \delta \}} | Z_k - Z_n| > \varepsilon \right) < \varepsilon,
\end{equation}
a condition called {\em uniform continuity in probability}, Anscombe shows that $ Z_{\tau_a} \stackrel{d}{\to} Z $, as $ a \to \infty $. 
His results have been adopted to many applications and generalized considerably.  
For example, when strengthening (\ref{AnscCond}) to 
\[
  P\left( \left| \frac{\tau_a}{a \lambda} - 1 \right| > \delta_a \right) = O( \delta_a^{1/2} ),
\]
where $ \lambda^{-1} \le \delta_a \to 0 $, as $ a \to \infty $, then a Berry-Esseen result holds true, that is the distribution
of $ S_{\tau_a}/ (\sigma \tau_a) $ converges uniformly to the standard normal distribution function, cf. \citet[Theorem~2.7.3]{GoshEtAl1997}.
\citet{Gut1991} established Anscombe-type laws of the iterated logarithm by strengthening (\ref{UniformContinuity}) to
\begin{equation}
\label{UniformContinuity2}
 \sum_{n=1}^\infty P \left( \max_{\{ k : |k-n| < n \delta \}} | Z_k - Z_n| > \varepsilon \right) < \infty.
\end{equation}
For further extensions in this direction, e.g., to $U$-statistics, and applications we refer to \citet{GoshDasgupta1980}, \cite{Mukho1981},
and \citet{MukhoVik1985}, amongst others. Finally, it is known that Anscombe's central limit theorem stated in Theorem~\ref{ThAnscombe} extends to a functional central limit theorem with Brownian motion as the limit process; we refer to \cite{Billingsley1999}, \cite{Larsson2000} and \cite{Gut2009}, amongst others.

Particularly having in mind complex applications where concrete definitions of the random time horizon may be
unknown to the statistician when designing the sequential procedure, it is remarkable 
that the result holds true without {\em any} condition on the dependence of the increments of the partial sums in
Theorem~\ref{ThAnscombe},  i.e. $ \{ X_n : n \ge 1 \} $, and the family of stopping times
$ \{ \tau_a : a > 0 \} $. Even stopping times which analyze the random increments directly can be used
without affecting the asymptotic normality for  $ a \to \infty $. Indeed,
a standard example for a family $ \{ \tau_a : a > 0 \} $ satisfying Anscombe's condition (\ref{AnscCond}) is the
first passage time of the random walk related to an i.i.d. sequence $ X_1, X_2, \dots $ with common mean $ \mu \not= 0 $, 
\[
  \tau_a = \inf \{ T \in \N : S_T > a \}, \qquad a > 0,
\]
e.g., costs associated with the continuation of the sequential procedure,
where as in the above theorem and, with some abuse of the notation used in previous sections,
\[
  S_T = \sum_{i=1}^T X_i, \quad T \in \N.
\]
Then it is well known that
\[
  \frac{\tau}{a} \stackrel{a.s.}{\to} \lambda = \frac{1}{\mu},
\]
as $ a \to \infty $, cf. the proof of Lemma~2.9.2 in \cite{GoshEtAl1997}.

As a second important example let us consider the following sequential estimation setting discussed
by Anscombe  in his 1952 paper. That example also shows that Anscombe's results address a deficiency of
sequential procedures such as the sequential probability ratio test, namely the fact that an open-ended 
stopping rule which is applied in order to stop sampling as soon as it is possible to decide in which
subset of the parameter space a parameter lies may lead to samples sizes which are too small for estimation
of parameters, cf. the discussion in \citet[Ch. 5]{Siegmund1985}. That early-stopping issue can be approached as follows.
Aiming at estimating  a parameter $ \theta $ from the data we sample until an estimate of the
estimator's dispersion is less or equal some threshold $ c_a $, where $ c_a \downarrow 0 $ as $ a \to \infty $, 
and then estimate the parameter by an estimator $ \wh{\theta}_n $ which is assumed to converge in distribution after
standardization. Given the family 
\[
  \tau_a = \inf \{ n \in \N : \widehat{\operatorname{s.d.}}( \widehat{\theta}_n ) \le c_a \}, \qquad a > 0,
\]
defined in this way satisfies
\[ 
  \tau_a / \tau^*_a \to 1, \qquad \text{in probability, as $ a \to \infty $},
\] 
where 
\[
  \tau^*_a= \inf \{ n \in \N : \sqrt{ \Var( \widehat{\theta}_n ) } \le c_a \}, \qquad a > 0,
\]
is the corresponding least sample
size such that the true dispersion of the estimator is less or equal than $ c_a $, Anscombe shows that
the above sequential sampling scheme yields an estimator which inherits the asymptotic distribution
with the true dispersion replaced by $ c_a $. This means, one may achieve estimation with given small
accuracy $ c_a $.

Our interest is now to extend the weak convergence results for the cross-validation criterion to the case of a random time horizon. 
We shall see that the time horizon can indeed be replaced by a family of random indices under quite general conditions, but the
interpretation differs: By randomizing the time horizon in such a controlled way instead of fixing it at a large value, 
we may ensure certain properties, such as a guaranteed accuracy of some estimator of interest, 
in the case that a (closed-end) stopping rule did not lead to a signal before the time horizon.
This is particularly beneficial when monitoring a time series automatically and expecting a signal indicating a change only
with low probability, such that the typically outcome is that the procedure runs until time $T$. Having reached the time horizon $T$,
one might be interested in analyzing the sample obtained in this way using classic methods of estimation and testing.

Another motivation is that there may be events which should trigger immediate termination of a monitoring procedure. As an example, suppose
one monitors the mean of an investment portfolio by applying the procedure $ S_T^- $ to the (discounted) value process of the
portfolio, 
in order to get an alarm if the investment strategy performs poor. But in case that the associated risk $ r_t $, which can be measured by a dispersion statistics such as the standard deviation or by value-at-risk, cf. \citet{Steland2012Book}, or the risk of some other
important financial variable exceeds an upper risk limit, one should terminate immediately. This gives rise to a family of 
stopping times such as  $ \tau_a = \inf \{ n < T'+1 : r_n > \overline{r}_a \} $, where $ T' = T $ or $ T' = \infty $, and 
$ \overline{r}_a $ is the upper risk limit parameterized by $ a > 0 $.

In what follows, we shall now discuss a random time horizon limit theorem for the cross-validation process, 
which is affected when applying a Anscombe-type random stopping
procedure to the time horizon of the detectors $ S_T^+ $ and $ S_T^- $ defined in (\ref{STplus}) and (\ref{STminus}), 
respectively. However, it will turn out that the arguments go through for many other processes as well.

Recall that the cross-validation process $ C_{T,s}(h) $ is dominated by the process $ L_T(s) $ and satisfies
\[
  \wt{C}_T(s) = T C_T(s)   \Rightarrow \calL_\xi(s),
\]
as $ T \to \infty $. We are interested in the randomly stopped sequential processes
\[
  \wt{C}_{\tau_a} = \wt{C}_T \bigr|_{T=\tau_a}, \qquad a > 0,
\]
and
\[
  \xi^*_{\tau_a} = \xi_T^*  \bigr|_{T=\tau_a}, \qquad a > 0.
\]

The following main result of this section provides
an Anscombe-type theorem for $ \wt{C}_{\tau_a} $. Its proof is based on the key observation that in our setting the
random stopping can be interpreted as a random change of time. 

\begin{theorem} 
\label{AnscombeTypeTheorem}
Let $ \{ \tau_a : a > 0 \} $ be a family of random variables taking values in $ \N $ such that condition (\ref{AnscCond}) holds
true for some $ \lambda \le 1 $. Then the  process $ \{ \wt{C}_{\tau_a} : a > 0 \} $ with random time horizon $ \tau_a $
satisfies the functional central limit theorem
\[
  \wt{C}_{\tau_a}(s) \Rightarrow \calL_{\xi}( \lambda s ),
\]
as $ a \to \infty $, in $ D([s_0,1];\R)$. Further, if $ \Xi = \{ \xi^1, \dots, \xi^M \} $ as in Subsection~\ref{CrossValidatedPr},
\begin{equation}
\label{CrossValLambda}
  \xi_{\tau_a}^*(s) \Rightarrow \argmin_{\xi \in \Xi} \calL_\xi( \lambda s ),
\end{equation}
as $ a \to \infty $, in $D([s_0,1];\R) $, provided the assumptions imposed there hold true.
\end{theorem}

\begin{proof}  The proof draws on \citet{Billingsley1999}, but our setting differs slightly. For constants $ A \le B $ and
$ C \le D $ let $ D_0( [A,B]; [C,D] ) $ 
denote the set of those elements $ f $ of $ D( [A,B]; [C,D] ) $ that are nondecreasing and satisfy $ C \le f(t) \le D $ for all $ t $;
$ C([A,B];[C,D]) $ is defined accordingly.
Introducing the parameter $ T' = \lceil a \rceil $, $ a > 0 $, we can embed $ \wt{C}_{\tau} $ into the sequence 
$ \{ \wt{C}_{T'} : T' \ge 1 \} $ of processes via the crucial identity 
\[
  \wt{C}_{\tau_a}(s) = \wt{C}_{T'} \left( \frac{ \tau_a }{ T' } s \right), \qquad s \in [s_0,1], \quad T' \in \N,
\]
that is
\begin{equation}
\label{CrucialRelation}
  \wt{C}_{\tau_a} = \wt{C}_{T'} \circ \Phi_{T'}, \qquad T' \in \N,
\end{equation}
where 
\[
  \Phi_{T'}(s) = \frac{\tau_a}{T'} s, \qquad s \in [s_0,1], \  T' \in \N.
\]
Notice that $ \frac{\tau_a}{a} \to \lambda $, as $ a \to \infty $, and $ \lambda \le 1 $ imply that
$ \Phi_{T'} $ takes values in $ C( [s_0,1]; [\lambda s_0 - \varepsilon,1] ) $ for large enough $ T' $,
given any arbitrary small $ \varepsilon > 0 $.
The result now follows easily from the representation (\ref{CrucialRelation}). We have
the joint weak convergence 
\[
  ( \wt{C}_{T'}, \Phi_{T'} ) \Rightarrow (\calL_\xi, \Phi ),
\] 
as $ T' \to \infty $, in the product space $ D([\lambda s_0-\varepsilon,1]; \R) \otimes D_0([s_0,1];[\lambda s_0-\varepsilon,1]) $, where
$ \Phi \in C_0( [s_0,1]; [\lambda s_0,1] ) $ is the multiplication with $ \lambda $, i.e. $ \Phi(s) = \lambda s $
for $ s \in [s_0,1] $. Indeed,
$ \Phi_{T'} $ converges weakly to the non-random continuous element $ \Phi $, since by
virtue of Anscombe's assumption (\ref{AnscCond}) 
\[
  \sup_{s \in [s_0,1]} | \Phi_{T'}(s) - \Phi(s) | \le \left| \frac{\tau_a}{a} - \lambda \right| \stackrel{P}{\to} 0,
\]
as $ a \to \infty $, which implies $ \Phi_{T'} \Rightarrow \Phi $, as $ T' \to \infty $.
Thus the result follows by an 
application of the continuous mapping theorem, since the composition of mappings
is a continuous functional, cf. \citet[p.151]{Billingsley1999}, and we can conclude that
\[
  \wt{C}_{\tau_a}(s) = \wt{C}_{T'} \circ \Phi_{T'}(s) \Rightarrow \calL_\xi \circ \Phi (s) = \calL_\xi( \lambda s),
\]
as $ a \to \infty $. The proof of (\ref{CrossValLambda}) is left to the reader.
\end{proof}

\begin{remark} The above result and its method of proof deserve some discussion.
\begin{itemize}
\item[(i)] An inspection of the proof of Theorem~\ref{AnscombeTypeTheorem} reveals that 
the arguments carry over to any empirical process $ X_T(s) $, particularly partial sum processes, such that the (functional) dependence on $ T $ and $ s $ is via multiplication $ T s $.
\item[(ii)] By either restricting the domain $ [s_0,1] $ to $ [s_0,1/\lambda] $ or taking the natural
extension of the limit theorems of the previous section to the spaces $ D([A,B];\R) $ for $ [A,B] \subset [s_0,\infty) $, 
one may easily generalize the result to an arbitrary limit $ \lambda \in (0, \infty) $ of $ \tau_a / a $.
\item[(iii)] The proof relies on the joint convergence of the process of interest, $ \wt{C}_{T'} $, and the
transformations, $ \Phi_{T'} $, which holds true if $ \Phi_{T'} $ converges to a non-random
function. The latter is guaranteed by condition (\ref{AnscCond}), which already appeared
in \citet{Anscombe1952}. However, the more general condition
\begin{equation}
\label{ConditionAnscombe2}
  \frac{ \tau_a }{ a } \stackrel{d}{\to} \Lambda, 
\end{equation}
as $ a \to \infty $, for some random variable $ \Lambda $, requires an explicit proof of the
joint weak convergence. This may require much more knowledge on $ \wt{C}_{T'} $, the 
definition of $ \tau_a $ and the dependence between both. Only in the case that 
$ \{ \wt{C}_{T'} : T' \ge 1 \} $ and $ \{ \tau_a : a > 0 \} $ are independent, the joint
weak convergence again follows. 
\end{itemize}
\end{remark}

Our discussion suggests to formulate the following corollary for the important special case that
the random experiment conducted to determine the time horizon is independent from the
observations, in order to extend the scope of our results to families of stopping times satisfying
(\ref{ConditionAnscombe2}).

\begin{corollary} Let $ \{ \tau_a : a > 0 \} $ be a family of random variables taking values in $ \N $.
If $ \{\tau_a : a > 0 \} $ is independent from $ \{ X_{Tn} : 1 \le n \le T, T \ge 1 \} $ and satisfies
\[
  \frac{\tau_a}{a} \stackrel{d}{\to} \Lambda,
\]
as $ a \to \infty $, for some random variable $ \Lambda $ taking values in $ (0,1] $,
then the randomly stopped process $ \wt{C}_{\tau_a} $ satisfies
\[
  \wt{C}_{\tau_a}(s) \Rightarrow \calL_{\xi}( s \Lambda )
  = -2  \int_0^{s\Lambda} 
     \frac{ \int_0^u K( \xi(u-v) ) \, d B_\delta^\sigma(v) }{ \int_0^u K( \xi(u-v) ) \, dv } \, d B_\delta^\sigma(u)
\]
as $ a \to \infty $. Further, given the assumptions imposed in Subsection~\ref{CrossValidatedPr}, then
\[
  \xi_{\tau_a}^*(s) \Rightarrow \argmin_{\xi \in \Xi} \calL_\xi( \Lambda s ),
\]
as $ a \to \infty $, in $D([s_0,1];\R) $, where $ \Xi = \{ \xi^1, \dots, \xi^M \} $. 
\end{corollary}

\section*{ACKNOWLEDGEMENTS}

Parts of this paper were written during research stays at the University of Manitoba, Canada, and Wroc\l aw University of Technology, Poland.
I thank the editor, Professor Nitis Mukhopadhyay, for inviting me to contribute this article.



\end{document}